\documentclass[10pt,a4paper]{amsart}
\usepackage{amssymb,amsmath,amsthm}
\usepackage{tikz}
\usepackage{tikz-cd}
\usepackage{enumitem}
\usepackage{amsaddr}
\usepackage{etoolbox}
\usepackage[utf8]{inputenc}
\usepackage{hyperref}
\usepackage{circuitikz}
\usepackage{float}
\usepackage{quiver}
\usepackage{tensor}
\usepackage{mathtools}
\usetikzlibrary{arrows,shapes,positioning}
\usepackage{subfig}
\usepackage{verbatim}

\usepackage{graphicx} 

\title{Combinatorial Insights on Tensor Power Decomposition in  $U_q(\mathfrak{sl}_2)-$module}
\author{Vinit Sinha}
\email{vinitsinha21feb@gmail.com}
\date{}

\begin{document}
\newtheorem{defn}{Definition}[section]
\newtheorem{definitions}[defn]{Definitions}
\newtheorem{lem}[defn]{Lemma}
\newtheorem{construction}[defn]{Construction}
\newtheorem{prop}[defn]{Proposition}
\newtheorem*{prop*}{Proposition}
\newtheorem{thm}[defn]{Theorem}
\newtheorem{cor}[defn]{Corollary}
\newtheorem{claim}{Claim}[defn]
\newtheorem*{claim*}{Claim}
\newtheorem{algo}[defn]{Algorithm}
\theoremstyle{remark}
\newtheorem{rem}[defn]{Remark}
\theoremstyle{remark}
\newtheorem{remarks}[defn]{Remarks}
\theoremstyle{remark}
\newtheorem{notation}[defn]{Notation}
\theoremstyle{remark}
\newtheorem{exmp}[defn]{Example}
\theoremstyle{remark}
\newtheorem{examples}[defn]{Examples}
\theoremstyle{remark}
\newtheorem{dgram}[defn]{Diagram}
\theoremstyle{remark}
\newtheorem{fact}[defn]{Fact}
\theoremstyle{remark}
\newtheorem{illust}[defn]{Illustration}
\theoremstyle{remark}
\newtheorem{que}[defn]{Question}
\numberwithin{equation}{section}

\newcommand{\N}{\mathbb{N}} 
\newcommand{\R}{\mathbb{R}}
\newcommand{\C}{\mathbb{C}}
\newcommand{\Z}{\mathbb{Z}}
\newcommand{\SL}{\mathfrak{sl}_{2}}
\newcommand\UU{\mathcal{U}}

\maketitle
\allowdisplaybreaks

\begin{abstract}
    Let $V(1)$ be the natural representation of $U(\SL).$ The multiplicities of $V(k)$ in $V(1)^{\otimes N}$ have multiple interpretations in combinatorics. In this paper, we investigate one such combinatorial interpretation of their multiplicities. Furthermore, we extend our analysis to the two-dimensional representation $T(1)$ of the restricted quantum enveloping algebra $U_q^{\mathrm{res}}(\SL)$. By employing combinatorial techniques, this study aims to elucidate the multiplicity of $T(k)$ within $T(1)^{\otimes N}$ while also offering a comprehensive analysis of the asymptotic behavior of these multiplicities as the parameter $N$ approaches infinity.
\end{abstract}

\section{Introduction}

Quantum groups, first introduced by Drinfeld\cite{drinfeld1986quantum} and Jimbo\cite{jimbo1985q}, are a class of Hopf algebra which arise as a deformation to the quantized universal enveloping algebra. Quantum groups have become a fundamental topic in modern mathematical physics and representation theory providing a fascinating generalization of classical Lie algebras and enveloping algebras, exhibiting intriguing algebraic and geometric structures. 

In this paper, we delve into the representation theory of one parameter deformation of the universal enveloping algebra $U(\SL)$, focusing on the combinatorial aspects of the tensor power decomposition for tilting modules\cite[\S~11.3]{chari1995guide} of $U_q^{\mathrm{res}} (\SL)$\cite[\S~9.3]{chari1995guide}. Specifically, we investigate the multiplicity of $T(k)$ in the tensor product $T(1)^{\otimes N}$, where $T(k)$ is the 
unique (up to isomorphism) indecomposable tilting module of maximal weight $k$\cite[\S~11.3.4]{chari1995guide}.

\begin{figure}[H]
\centering
\begin{subfloat}[Example 1] {
\label{ex:1}
\centering
    \begin{tikzpicture}
   
            \draw[very thin,color=gray!15,step=1] (0,0) grid (5,5);
        \draw[-latex] (0,0) -- (5.2,0) node[right] { };
        \draw[-latex] (0,0) -- (0,5.2) node[above] {};
        \coordinate [label=below left:O] (a) at (0,0);
        \foreach \i in {1,2,...,5}
        \draw[gray!65] (\i,.1)--(\i,-.1) node[below] {$\i$};
        \foreach \i in {1,2,3,4,5}
        \draw[gray!65] (.1,\i)--(-.1,\i) node[left] {$\i$};
        \draw[thick] (0,0)--(4,4){};
        \draw[thick] (4,4)--(5,3){};
    \end{tikzpicture}
}
\end{subfloat}
\hspace{10pt}       
\begin{subfloat}[Example 2] {
\centering
    \begin{tikzpicture}
    
        \draw[very thin,color=gray!15,step=1] (0,0) grid (5,5);
        \draw[-latex] (0,0) -- (5.2,0) node[right] { };
        \draw[-latex] (0,0) -- (0,5.2) node[above] {};
        \coordinate [label=below left:O] (a) at (0,0);
        \foreach \i in {1,2,...,5}
        \draw[gray!65] (\i,.1)--(\i,-.1) node[below] {$\i$};
        \foreach \i in {1,2,3,4,5}
        \draw[gray!65] (.1,\i)--(-.1,\i) node[left] {$\i$};
        \draw[thick] (0,0)--(1,1){};
        \draw[thick] (1,1)--(2,0){};
        \draw[thick] (2,0)--(3,1){};
        \draw[thick] (3,1)--(4,0){};
        \draw[thick] (4,0)--(5,1){};
    \end{tikzpicture}
}     
\end{subfloat}
\caption*{Catalan paths}

\end{figure}

We first start by looking at the tensor product decomposition of the irreducible representation of $U(\SL)$, specifically looking at the multiplicities of irreducible representations in the direct sum decomposition of $V(1)^{\otimes N}$ where $V(1)$ is the unique $2-$dimensional irreducible representation of $U(\SL)$. We have established an alternate method of finding $t(k,N)$ ($t(k,N)$ is defined as the multiplicity of $V(k)$ in $V(1)^{\otimes N}$) using the enumeration of Catalan paths and Dyck paths\cite[\S~10.8]{bona2015handbook}. Then we have further shown that $t(0,N)$ is equal to the enumeration of all the \emph{noncrossing (complete) matching on $N-$vertices}.

\begin{figure}[H]
    \centering
\begin{subfloat} {
\label{ex:1}
\centering
    \begin{tikzpicture}
            \foreach \i in {1,2,...,6}
        \draw[black!65] (\i-1,.1)--(\i-1,-.1) node[below] {$\i$};
            \draw[very thin](0,0)--(5,0){};
            \draw[very thick] (1,0) arc (0:180:0.5);
            \draw[very thick] (3,0) arc (0:180:0.5);
            \draw[very thick] (5,0) arc (0:180:0.5);
    \end{tikzpicture}
}
\end{subfloat}
\hspace{10pt}       
\begin{subfloat} {
\begin{tikzpicture}
            \foreach \i in {1,2,...,6}
        \draw[black!65] (\i-1,.1)--(\i-1,-.1) node[below] {$\i$};
            \draw[very thin](0,0)--(5,0){};
            \draw[very thick] (1,0) arc (0:180:0.5);
            \draw[very thick] (5,0) arc (0:180:1.5);
            \draw[very thick] (4,0) arc (0:180:0.5);
    \end{tikzpicture}
}     
\end{subfloat}
\hspace{10pt}       
\begin{subfloat} {
\begin{tikzpicture}
            \foreach \i in {1,2,...,6}
        \draw[black!65] (\i-1,.1)--(\i-1,-.1) node[below] {$\i$};
            \draw[very thin](0,0)--(5,0){};
            \draw[very thick] (3,0) arc (0:180:0.5);
            \draw[very thick] (5,0) arc (0:180:2.5);
            \draw[very thick] (4,0) arc (0:180:1.5);
    \end{tikzpicture}
}     
\end{subfloat}
\hspace{10pt}       
\begin{subfloat} {
\begin{tikzpicture}
            \foreach \i in {1,2,...,6}
        \draw[black!65] (\i-1,.1)--(\i-1,-.1) node[below] {$\i$};
            \draw[very thin](0,0)--(5,0){};
            \draw[very thick] (2,0) arc (0:180:0.5);
            \draw[very thick] (5,0) arc (0:180:2.5);
            \draw[very thick] (4,0) arc (0:180:0.5);
    \end{tikzpicture}
}     
\end{subfloat}
\caption*{Examples of noncrossing matching}
\end{figure}
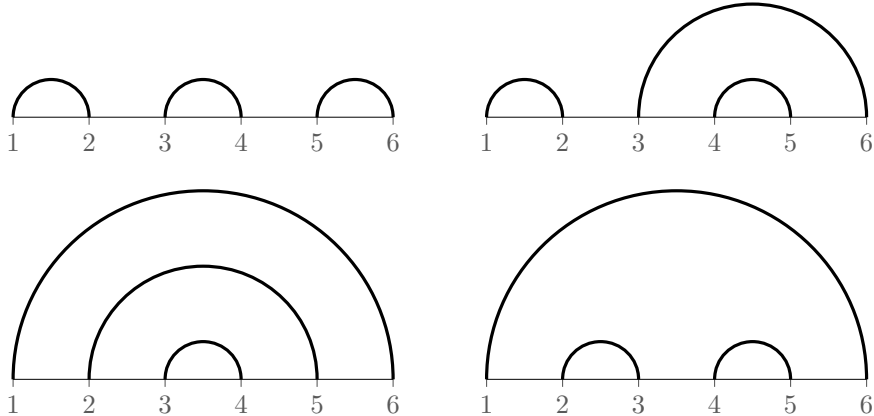

We further delve into our analysis by exploring the quantum enveloping algebra $U_q(\mathfrak{sl_2})$. In particular, our focus lies on the tilting modules within the restricted specialization of this algebra, denoted as $U_q^{\mathrm{res}}(\mathfrak{sl}_2)$. Lachowska et al.\cite{anna2024multiplicities} have given an explicit formula for the direct sum decomposition of the tensor power of indecomposable tilting modules. We have used that result to find the coefficient of tilting modules in the direct sum decomposition of $T(1)^{\otimes N}$. We have also established a correspondence between the coefficient of $T(k)$ in $T(1)^{\otimes N}$ with the enumeration of some particular class of Catalan paths. We then note a relation between \emph{tridiagonal Toeplitz matrices} \cite{noschese2013tridiagonal} and coefficients of the direct sum decomposition in $T(1)^{\otimes N}$. Using that relation, we came up with the following result.

\begin{prop*}
    Let $p$ be a function such that $T(1)^{\otimes N} = \bigoplus\limits_{k \in \Z_{\geq 0}} T(k)^{\oplus p(k,N)}$. Then for any $k \in \{0,1,\cdots , l-2\}$,

    $$ p(k,N) = \frac{2^{N+1}}{l}\sum_{s=1}^{l-1} \cos^N\left( \frac{s\pi}{l}\right)\sin\left(\frac{s\pi}{l} \right)\sin\left(\frac{s(k+1)\pi}{l}\right) \text{ for every $N \in \N$}. $$
\end{prop*}

\noindent{}Finally, we have used the above result to investigate the asymptotic behaviour of $p$ as $N \to \infty$. 

The structure of the paper is as follows. In Section \ref{Intro to sl2}, we have recalled the basic definitions and results related to the $\SL-$algebra. In section \ref{Catalan paths}, we have looked at Catalan paths and established a way of finding $t(k,N)$ by certain enumerations of these Catalan paths. In Section \ref{noncrossing matchings}, we have looked at \emph{noncrossing matchings on $2N$ vertices} and how their enumerations result in $t(0,2N)$. Section \ref{Quantum group} provides a brief overview of the background and key concepts related to quantum groups, Weyl filtration and tilting modules that are relevant to our discussion. In Section \ref{Bounded Catalan paths}, we go on to establish the connection between $p(k,N)$ and the enumeration of a specific type of Catalan paths (we call it \emph{Partially $l-$Bounded Catalan Paths}). Finally, Section \ref{Asymptotic behavior} establishes the connection between the tridiagonal Toeplitz matrix and $p(k,N)$ to prove the above result. We then use it to describe the asymptotic behaviour of $p(k,N)$ as $N \to \infty$.

In this report, we have used $\N_0$ as the set of non-negative numbers and $\N$ as the set of positive integers.

\subsection*{Acknowledgements}

I extend my sincere gratitude to Dr Anna Lachowska for her invaluable supervision and insightful discussions throughout the course of this research. This work was carried out during an exchange programme at EPFL. I wish to express my heartfelt thanks to EPFL for their generous hospitality during the development of this paper.

\section{Finite dimensional representations of $\SL$}
\label{Intro to sl2}
Throughout this report, we will fix the ground field as $\C$. The Lie algebra $\SL$ is generated by three basis elements $\{ e, f, h\}$ with given relations. $$[e,f] = h\hspace{3mm} [h,e] = 2e\hspace{3mm} [h,f] = -2f$$

Thus note that $\UU  (\SL) = \C\langle e,f,h \rangle / \langle ef-fe-h, he-eh-2e, hf-fh+2f \rangle$. We want to look at the decomposition of finite dimensional representations of $\SL$ into indecomposable representations. Before doing that, we must classify all the irreducible and indecomposable finite-dimensional representations. 

\begin{thm}\label{thm: simple_unique}\cite[Theorem~V.4.4] {kassel2012quantum} 
For any $n \in \N \cup \{0\}$, there exists a unique (up to isomorphism) simple $n-\mathrm{dimensional}$ $\UU (\SL)-$module.
\end{thm}

Thus for $n \in \N \cup \{0\}$ we  define $V(n)$ as the simple $n-$dimensional $\UU(\SL)-$module. We will now look at a result that classifies all the finite-dimensional $\UU (\SL)-$modules.

\begin{thm}\label{thm: semisim}\cite[Theorem V.4.6]{kassel2012quantum}
Any finite-dimensional $\UU (\SL)-$module is semisimple.
\end{thm}

Thus Theorem \ref{thm: simple_unique} and \ref{thm: semisim} implies that if $V$ is a simple finite-dimensional module of $\UU(\SL)$ then there exists a finite (possibly repeating) sequence $(n_i : i \in \N)$ such that
$$V \cong \bigoplus\limits_{i\in \N} V(i)^{\oplus n_i}.$$

We will focus on the direct sum decomposition of the tensor product of two finite-dimensional modules. By Theorem \ref{thm: semisim}, we only need to look at the tensor product of two simple finite-dimensional $\UU (\SL)-$modules. The direct sum decomposition of the tensor product of two simple finite-dimensional is given by \emph{Clebsch-Gordan formulae}. 

\begin{prop}[Clebsch-Gordan formulae]\label{clebsch-gordan} For any $n,m \in \N$, if $n\geq m$ then

$$ V(n) \otimes V(m) = \bigoplus\limits_{i=0}^m V(n-m+2i). $$

\end{prop}

We are interested in on the direct sum decomposition of $V(1)^{\otimes n}$ as a direct sum decomposition of simple modules. Define $V(1)^{\otimes 0 } := V(0)$. Since $V(1)^{\otimes n} = \bigoplus\limits_{i\in \N_0} V(i)^{\oplus n_i}$ for some sequence $(n_i : i \in \N_0)$, define a sequence $(t(k,N): N \in \N, k \in \N_0)$ such that for every $N \in \N$

$$V(1)^{\otimes N} = \bigoplus_{k\in \N} V(k)^{\oplus t(k,N)} .$$

We fix $t(-1,N) = 0$ for every $N \in \N_0$.

\begin{prop}\label{prop: recursion of t}
    For every $N \in \N$ and $k \in \N_0$, $$t(k,N) = t(k-1,N-1)+ t(k+1,N-1)$$
    $$t(k,0) = \begin{dcases*}
1
   & if  $k = 1$\,, \\[1ex]
0
   & otherwise\,.
\end{dcases*}$$
\end{prop}

\begin{proof}
    Since $V(1)^{\otimes 0} = V(0)$, we get 

    $$t(k,0) = \begin{dcases*}
1
   & if  $k = 1$\,, \\[1ex]
0
   & otherwise\,.
\end{dcases*}$$
    
    Moreover, Note that,
    \begin{align*}
        V(1)^{\otimes N} &= V(1)^{\otimes (N-1)} \otimes V(1) = \bigoplus\limits_{k \in \N_0} V(k)^{\oplus t(k,N-1)} \otimes V(1) \\&= V(1) \oplus \bigoplus\limits_{k \in \N} (V(k-1) \oplus V(k+1))^{\oplus t(k,N-1)} \\&= \bigoplus\limits_{k \in \N} V(k-1)^{\oplus t(k,N-1)} \oplus \left( V(1)\oplus\bigoplus\limits_{k \in \N} V(k+1)^{\oplus t(k,N-1)}\right) \\&= \bigoplus\limits_{k \in \N_0} V(k)^{\oplus t(k+1,N-1)} \oplus \bigoplus\limits_{k \in \N_0} V(k)^{\oplus t(k-1,N-1)} \\&= \bigoplus\limits_{k \in \N_0} V(k)^{\oplus t(k+1,N-1)} \oplus \bigoplus\limits_{k \in \N_0} V(k)^{\oplus t(k-1,N-1)} \\ \Longrightarrow \bigoplus\limits_{k\in \N_0} V(k)^{\oplus t(k,N)} &= \bigoplus\limits_{k \in \N_0} V(k)^{\oplus (t(k+1,N-1) + t(k-1,N-1))}
    \end{align*}
    Therefore, $t(k,N) = t(k-1,N-1) + t(k+1,N-1)$.
\end{proof}
\begin{cor}\label{cor: finite sum for t terms}
    If $k \geq 1, N \in \N$ and $k' \in \{ 1,2,\cdots, k \}$, then $t(k,N) = \sum\limits_{i=0}^{N-1} t(k'-1, i)t(k-k', N-i-1)$
\end{cor}
\begin{proof}
    We will prove the Corollary by induction on $N \in \N$.

    \textbf{Base Case.} If $N = 1,$ then $\sum\limits_{i=0}^{N-1} t(k'-1, i)t(k-k', N-i-1) = t(k'-1,0)t(k-k',0)$ we will take two cases.
    \begin{itemize}
        \item If $k =1,$ then $k' = 1.$ Therefore, $$\sum\limits_{i=0}^{N-1}t(k'-1,i)t(k-k', N-i-1) = t(0,0)^2 = 1 = t(k,N).$$
        \item If $k >1,$ then since $(k-k') + (k'-1) = k-1 > 0$ we get that either $(k-k')>0$ or $(k'-1)>0$. Therefore $$\sum\limits_{i=0}^{N-1}t(k'-1,i)t(k-k', N-i-1) = t(k'-1,0)t(k-k',0) = 0 = t(k,N).$$
    \end{itemize}

    \textbf{Induction Step.} If $N>1$, then Proposition  \ref{prop: recursion of t} yields $t(k,N) = t(k-1,N-1) + t(k+1, N-1)$. We have two cases.
    \begin{itemize}
\item If $k = k'$, then induction hypothesis yields $t(k+1,N-1) = \sum\limits_{i=0}^{N-2} t(1,i)t(k-1, N-i-2)$. Therefore,
\begin{align*}
     t(k,N) &= t(k-1,N-1) +  \sum\limits_{i=0}^{N-2} t(1,i)t(k-1, N-i-2)\\
     &= t(0,0)t(k-1,N-1) +  \sum\limits_{i=1}^{N-1} t(1,i-1)t(k-1, N-i-1)\\
     &= t(0,0)t(k-1,N-1) + \sum\limits_{i=1}^{N-1} t(0,i)t(k-1, N-i-1)\tag{by Proposition \ref{prop: recursion of t}}\\&= \sum_{i=0}^{N-1} t(k-k',i)t(k'-1, N-i-1)
\end{align*}

\item If $k>k',$  then by induction hypothesis $t(k-1,N-1) = \sum\limits_{i=0}^{N-2} t(k'-1,i)t(k-1-k', N-i-2) $ and $t(k+1, N-1) = \sum\limits_{i=0}^{N-2} t(k'-1,i)t(k+1-k', N-i-2)$. Therefore,
\begin{align*}
    t(k,N) &= \sum\limits_{i=0}^{N-2} t(k'-1,i) \left( t(k+1-k', N-i-2) + t(k-1-k', N-i-2) \right)\\
    &= \sum\limits_{i=0}^{N-2} t(k'-1,i)t(k-k', N-i-1) \tag{by Proposition \ref{prop: recursion of t}}.
\end{align*}
\end{itemize}
\end{proof}

In the following sections, we will look at how $t(k,N)$ can be obtained by certain combinatorial objects like Catalan paths and non-crossing matching on $2N$ points. 

\section{Catalan paths and its correspondence to $t(k,N)$}\label{Section: catalan paths and t(k,N)}
\label{Catalan paths}
We will now focus on Catalan paths in $\mathbb \Z^2$. This section will show how counting these Catalan paths gives us the coefficients $t(k, N)$. Catalan paths \cite[\S~10.8]{bona2015handbook} are finite paths in $\Z^2$ starting from $(0,0)$ with steps in $(1,1)$ and $(1,-1)$ above the X-axis.

\begin{figure}[H]
\centering
\begin{subfloat}[Example 1] {
\label{ex:1}
\centering
    \begin{tikzpicture}
   
            \draw[very thin,color=gray!15,step=1] (0,0) grid (5,5);
        \draw[-latex] (0,0) -- (5.2,0) node[right] { };
        \draw[-latex] (0,0) -- (0,5.2) node[above] {};
        \coordinate [label=below left:O] (a) at (0,0);
        \foreach \i in {1,2,...,5}
        \draw[gray!65] (\i,.1)--(\i,-.1) node[below] {$\i$};
        \foreach \i in {1,2,...,5}
        \draw[gray!65] (.1,\i)--(-.1,\i) node[left] {$\i$};
        \draw[thick] (0,0)--(2,2){};
        \draw[thick] (2,2)--(4,0){};
        \draw[thick] (4,0)--(5,1){};
    \end{tikzpicture}
}
\end{subfloat}
\hspace{10pt}       
\begin{subfloat}[Example 2] {
\centering
    \begin{tikzpicture}
    
        \draw[very thin,color=gray!15,step=1] (0,0) grid (5,5);
        \draw[-latex] (0,0) -- (5.2,0) node[right] { };
        \draw[-latex] (0,0) -- (0,5.2) node[above] {};
        \coordinate [label=below left:O] (a) at (0,0);
        \foreach \i in {1,2,...,5}
        \draw[gray!65] (\i,.1)--(\i,-.1) node[below] {$\i$};
        \foreach \i in {1,2,...,5}
        \draw[gray!65] (.1,\i)--(-.1,\i) node[left] {$\i$};
        \draw[thick] (0,0)--(1,1){};
        \draw[thick] (1,1)--(2,0){};
        \draw[thick] (2,0)--(4,2){};
        \draw[thick] (4,2)--(5,1){};
    \end{tikzpicture}
}     
\end{subfloat}
\caption*{Examples of Catalan paths}
\label{Slika:fig1}

\end{figure}

We will now give a formal definition of Catalan paths, represented through a function on $\Z$, such that $f(i)$ is the point of intersection of the path with $x=i$. For example, the Catalan path in Example 1 will be represented by a function $f: \{0,1,2,3,4,5\} \to \Z$ such that 

$$ f(0) = 0,\, f(1) = 1,\, f(2) = 2, \, f(3) = 1 ,\, f(4) = 0, \, f(5) = 1.$$

\begin{defn}[Catalan paths]\label{defn: Catalan paths}\cite[\S~10.8]{bona2015handbook}
For $N \in \N_0$, a function $f: \{0,1,\cdots ,N\} \to \Z$ is called a Catalan path (of length $N$) if $f(0) = 0$, $\mathrm{Img(f) \subset \N_0}$ and $\lvert f(i)-f(i-1)\rvert = 1$ for every $i \in \{1,2,\cdots, N\}$
\end{defn}

For simplicity, a Catalan path $f$ is said to end at $(N,k)$ if $\mathrm{dom}(f) = \{0,1,\cdots, N\}$ and $f(N) = k$. Since $f(i+1)-f(i) \in \{1,-1\}$ for any Catalan path $f:\{0,1,\cdots, N\}\to \Z$ and $i \in \{0,1,\cdots, N-1\}$, we can state the following Remark. 

\begin{rem}\label{rem: Parity of difference}
    For any Catalan path $f: \{0,1, \cdots, N\} \to \Z$ and $i,j \in \{0,1, \cdots, N\}$, $$j-i \equiv f(j)-f(i) \pmod{2}$$.
\end{rem}

The Catalan paths which ends at $(N, 0)$ for some $N \in \N$ are called \emph{Dyck paths} \cite[\S~10.8]{bona2015handbook}. For $m,n \in \Z$, define $\Bar{t} (m,n)$ as the number of Catalan paths of length $m$ from $(0,0)$ to $(m,n)$. Note that $\Bar{t} (m,n) = 0$ if $m<0$ or $n<0$.

\begin{prop} \label{prop: Recurrence of catalan paths}
    For any $k \in \N$ and $N\in \N_0$, $\Bar{t}(N,k) = \Bar{t}(N-1,k-1) + \Bar{t}(N-1, k+1)$.
\end{prop}
\begin{proof}
    Let $A(i,j)$ be the set of all Catalan path of length $i$ where $f(i) = j$. Since $\lvert f(N)-f(N-1) \rvert = 1$ for any $f \in A(N,k)$, we get that $f(N-1) \in \{ k-1, k+1\}$ for any $f \in A(N,k)$. Therefore $f \in A(N,k)$ iff $f \vert_{\{0,\cdots , N-1\}} \in A(N-1,k-1) \cup A(N-1,k+1)$. Hence $\Bar{t} (N,k) = \lvert A(N,k) \rvert = \lvert A(N-1 , k-1) \rvert + \lvert A(N-1, k+1) \rvert = \Bar{t}(N-1, k-1) + \Bar{t}(N-1, k+1)$.
\end{proof}

We will now look at our main results from this section.

\begin{prop}\label{prop: t(k,N) and Catalan paths}
    For every $k \in \N_0$ and $N \in \N$, $t(k,N) = \Bar{t} (N,k)$.
\end{prop}
\begin{proof}
We will prove this by induction on $N$.

\vspace{3mm}

\textbf{Base Case.} Note that there is only one Catalan path of length $1$ which is given below,

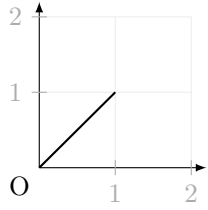
\begin{figure}[H]
    \centering
    \begin{tikzpicture}
    
        \draw[very thin,color=gray!15,step=1] (0,0) grid (2,2);
        \draw[-latex] (0,0) -- (2.2,0) node[right] { };
        \draw[-latex] (0,0) -- (0,2.2) node[above] {};
        \coordinate [label=below left:O] (a) at (0,0);
        \foreach \i in {1,2}
        \draw[gray!65] (\i,.1)--(\i,-.1) node[below] {$\i$};
        \foreach \i in {1,2}
        \draw[gray!65] (.1,\i)--(-.1,\i) node[left] {$\i$};
        \draw[thick] (0,0)--(1,1){};
    \end{tikzpicture}
    \label{fig:my_label}
    \caption*{Catalan path of length $1$}
\end{figure}
Therefore for every $k \in \N_0$, $\Bar t(1,k) = \delta_{1k}$ where $\delta_{ij}$ is the kronecker delta function. Since $V(1)^{\otimes 1} = V(1)$, we get that $t(k,1) = \delta_{1k}$ for every $k \in \N_0$. Therefore, $t(k,1) = \Bar{t} (1,k)$ for every $k \in \N_0$.
  
\vspace{3mm}

\textbf{Induction step.} Proposition \ref{prop: recursion of t} yields that $t(k,N) = t(k-1,N-1) + t(k+1, N-1)$. We will take two cases,
\begin{itemize}
    \item If $k =0$, then $t(0,N) = t(1,N-1)$. By induction hypothesis $t(0,N) = \Bar{t}(N-1,1)$. Proposition \ref{prop: Recurrence of catalan paths} yields $\Bar{t}(N,0) = \Bar{t}(N-1,1)+ \Bar{t}(N-1,-1) = \Bar{t}(N-1,1) = t(0,N)$.

    \item If $k>0$, then induction hypothesis yields $t(k,N) = \Bar{t}(N-1,k-1) + \Bar{t}(N-1,k+1)$. Thus Proposition \ref{prop: Recurrence of catalan paths} yields $t(k,N) = \Bar{t}(N,k)$.
\end{itemize}
\end{proof} 

\begin{cor}
For any $k,N \in \N$, if $t(k,N) > 0$ then $N \equiv k \pmod{2}$.
\end{cor}
\begin{proof}
Proposition \ref{prop: t(k,N) and Catalan paths} yields $\Bar{t}(N,k) > 0$. Thus there exists a Catalan path $f: \{0,1 \cdots, N\} \to \Z$ such that $f(N) = k$. Remark \ref{rem: Parity of difference} yields $(N - 0) \equiv (f(N)-f(0)) \pmod 2$. Therefore $N \equiv k \pmod 2$.
\end{proof}

\section{Noncrossing matching and its relation to $t(k,N)$}
\label{noncrossing matchings}
We will now look at ways of connecting $2N$ points in the plane lying on a horizontal line by $N$  nonintersecting arcs, each connecting two points and lying above the line.

\begin{figure}[H]
    \centering
\begin{subfloat} {
\label{ex:1}
\centering
    \begin{tikzpicture}
            \foreach \i in {1,2,...,6}
        \draw[black!65] (\i-1,.1)--(\i-1,-.1) node[below] {$\i$};
            \draw[very thin](0,0)--(5,0){};
            \draw[very thick] (1,0) arc (0:180:0.5);
            \draw[very thick] (3,0) arc (0:180:0.5);
            \draw[very thick] (5,0) arc (0:180:0.5);
    \end{tikzpicture}
}
\end{subfloat}
\hspace{10pt}       
\begin{subfloat} {
\begin{tikzpicture}
            \foreach \i in {1,2,...,6}
        \draw[black!65] (\i-1,.1)--(\i-1,-.1) node[below] {$\i$};
            \draw[very thin](0,0)--(5,0){};
            \draw[very thick] (1,0) arc (0:180:0.5);
            \draw[very thick] (5,0) arc (0:180:1.5);
            \draw[very thick] (4,0) arc (0:180:0.5);
    \end{tikzpicture}
}     
\end{subfloat}

\hspace{10pt}       
\begin{subfloat} {
\begin{tikzpicture}
            \foreach \i in {1,2,...,6}
        \draw[black!65] (\i-1,.1)--(\i-1,-.1) node[below] {$\i$};
            \draw[very thin](0,0)--(5,0){};
            \draw[very thick] (3,0) arc (0:180:0.5);
            \draw[very thick] (5,0) arc (0:180:2.5);
            \draw[very thick] (4,0) arc (0:180:1.5);
    \end{tikzpicture}
}     
\end{subfloat}
\hspace{10pt}       
\begin{subfloat} {
\begin{tikzpicture}
            \foreach \i in {1,2,...,6}
        \draw[black!65] (\i-1,.1)--(\i-1,-.1) node[below] {$\i$};
            \draw[very thin](0,0)--(5,0){};
            \draw[very thick] (2,0) arc (0:180:0.5);
            \draw[very thick] (5,0) arc (0:180:2.5);
            \draw[very thick] (4,0) arc (0:180:0.5);
    \end{tikzpicture}
}     
\end{subfloat}
    \caption*{Examples}
    \label{fig:my_label}
\end{figure}
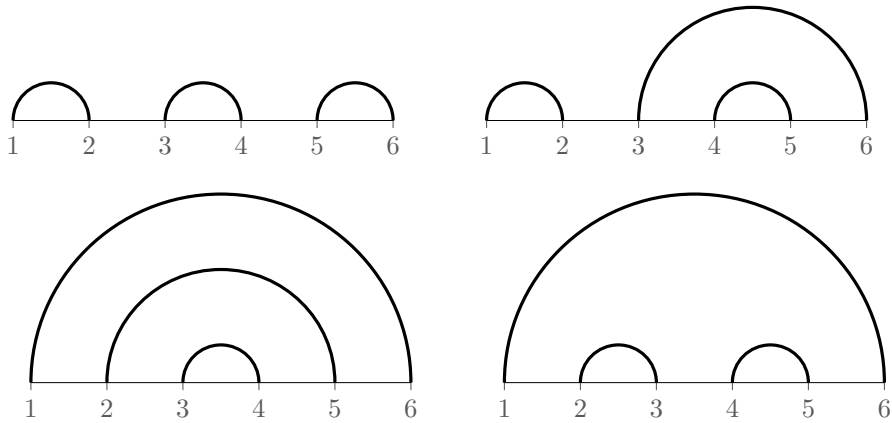

We will call such matchings on $2N$ vertices to be \emph{noncrossing  matchings on $2N$}. Note that the noncrossing matchings on $2N$ are uniquely represented by the set of its edges, e.g. the set $\{\{1,2\}, \{3,4\}, \{5,6\}\}$ uniquely represents this noncrossing matching:-

\begin{figure}[H]
    \centering
    \begin{tikzpicture}
            \foreach \i in {1,2,...,6}
        \draw[black!65] (\i-1,.1)--(\i-1,-.1) node[below] {$\i$};
            \draw[very thin](0,0)--(5,0){};
            \draw[very thick] (1,0) arc (0:180:0.5);
            \draw[very thick] (3,0) arc (0:180:0.5);
            \draw[very thick] (5,0) arc (0:180:0.5);
    \end{tikzpicture}
    \label{fig:my_label}
\end{figure}

\begin{rem} \label{rem: nested property}
    If $\{\{a_n, b_n\} : n \in \{1,\cdots, N\}\}$ represents a matching on $2N$, then $a_n + 1 \in \{a_1, \cdots, a_N\}\cup \{b_n\}$ and $b_n -1 \in \{b_1, \cdots, b_N\}\cup \{a_n\}$ for any $n \in \{1,\cdots,N\}$.
\end{rem}

We will now show how one can calculate some terms in the sequence $(t(k,N): k,N\in \N)$ by counting these noncrossing matchings.

\begin{prop}
    For any $N \in \N$, $\Bar{t}(2N,0)$ equals the number of noncrossing  matchings on $2N$.
\end{prop}
\begin{proof}
    In this proof, we will show a one-one correspondence between noncrossing matchings on $2N$ and Dyck paths ending at $(2N,0)$. Fix any Dyck path $f$ of length $2N$ ending at $(2N, 0)$. We will define a sequence $((a_n, b_n): n \in \{1,2, \cdots, N\})$ by induction such that it follow these properties:
    \begin{enumerate}
        \item For every $n \in \{ 1,2 \cdots, N\}$, $1 \leq a_n < b_n \leq 2N$;
        \item For every $n \in \{ 1,2 \cdots, N\}$, $\{a_1,a_2, \cdots ,a_{n-1}\} \subseteq \{1,2,\cdots, a_n \} \subseteq \{a_1,a_2, \cdots ,a_n\}\cup \{b_1,b_2, \cdots, b_n\}$;
        \item For every $n \in \{ 1,2 \cdots, N\}$, $a_n,b_n\notin\{a_1,a_2, \cdots ,a_{n-1}\}\cup \{b_1,b_2, \cdots, b_{n-1}\} $;
        \item For every $n \in \{ 1,2 \cdots, N\}$, $b_n$ is the least element greater than $a_n$ where $f(b_n)<f(a_n)$ (note that this will imply $f(b_n) = f(a_n)-1$).
    \end{enumerate}

\vspace{3mm}

    \textbf{Base case.} Define $a_1 := 1$ and take $b_1$ to be the least element greater than $1$ where $f(b_n)=0$.

    \textbf{Induction Step.} Define $a_n := min\left[\{1,2, \cdots, 2N\}\setminus \left(\{a_1,a_2, \cdots ,a_{n-1}\}\cup \{b_1,b_2, \cdots, b_{n-1}\}\right)\right]$ and take $b_n$ to be the least element greater than $a_n$ where $f(b_n)<f(a_n)$. Note that $(a_n,b_n)$ follow property $(1)$, $(2)$ and $(4)$.

    \noindent{}\textbf{Claim.} $(a_n,b_n)$ follows Property $(3)$.
    
    \noindent{}\textit{Proof of the claim.} Note that $a_n\notin\{a_1,a_2, \cdots ,a_{n-1}\}\cup \{b_1,b_2, \cdots, b_{n-1}\} $ follows from the definition of $a_n$. Since $(a_n,b_n)$ follows Property $(1)$ and $(2)$, we get $b_n \notin \{a_1,a_2, \cdots ,a_{n-1}\}$. Towards a contradiction assume $b_n \in \{b_1,b_2, \cdots, b_{n-1}\}$. Take $m \in \{1,2, \cdots, n-1\}$ such that $b_{m} = b_n$. Since $a_n\notin\{a_1,a_2, \cdots ,a_{n-1}\}\cup \{b_1,b_2, \cdots, b_{n-1}\}$, Property $(2)$ yields that $a_{m}<a_n$. Note that $f(a_{m}) = f(b_{m})+1 = f(b_n) + 1= f(a_n )$. Define an sequence $(m_k : k\in \N)$ such that $a_m<  a_{m_k} < a_{m_{k+1}}<a_n$ and $f(b_{m_k} ) = f(a_n)$ for every $k \in \N$.
    
    \textbf{Base case.} Since $f(a_m) = f(a_n)$ and $a_m < a_n$, Remark \ref{rem: Parity of difference} yields $a_n \geq a_m +2$. Property $(2)$ yields $a_m+1 \in \{a_1,a_2, \cdots ,a_{n-1}\}\cup \{b_1,b_2, \cdots, b_{n-1}\}$. Since $a_m < a_m + 1< a_n < b_n=b_m$, Property $(4)$ yields that $f(a_m + 1)\geq f(a_m)$, and hence $f(a_m+1) = f(a_m)+1$. We claim that $a_m + 1 \in \{a_1,a_2, \cdots ,a_{n-1}\}$ and define $m_1 \in \{ 1,\cdots, n-1\}$ such that $a_{m_1} = a_m + 1$. Towards a contradiction take $p \in \{1,2, \cdots, n-1\}$ such that $b_p = a_m + 1$. Thus $a_p\leq a_m < b_p $ and $f(a_m)< f(a_m+1) = f(b_p)< f(a_p)$, which is a contradiction to Property $(4)$. Note that Property $(4)$ also yields $f(b_{m_1}) = f(a_{m_1})-1 = f(a_m+1)-1 = f(a_m) = f(a_n)$.

    \textbf{Induction step.} Since $f(a_n) = f(b_{m_k} )< f(a_{m_k})$ and $b_{m_k} \neq  a_n$, Property $(4)$ yields $b_{m_k}<a_n$ and moreover Remark \ref{rem: Parity of difference} yields $a_n \geq b_{m_k}+2$. Since $a_m < b_{m_k} + 1< a_n < b_n=b_m $, Property $(4)$ yields $f(b_{m_k}+1)\geq f(a_m) = f(b_{m_k})$ and hence $f(b_{m_k} + 1) = f(b_{m_k})+1 = f(a_n)+1$. Following the similar arguments we made in the previous case, we can say that $b_{m_k} + 1 \in \{a_1,a_2, \cdots ,a_{n-1}\}$. Define $m_{k+1} \in \{1,2, \cdots, n-1\}$ such that $a_{m_{k+1}} = b_{m_k}+1$. Thus $f(b_{m_{k+1}}) = f(a_{m_{k+1}})-1 = f(b_{m_k}+1)-1 = f(a_n)$.

\vspace{1.5mm}
    
    Thus $(a_{m_k}: k \in \N)$ is an infinite increasing sequence in $\{ a_m+1, a_m+2, \cdots, a_n\}$, which is a contradiction as $\{ a_m+1, a_m+2, \cdots, a_n\}$ is a finite set. \hfill$\blacksquare$ 

\vspace{4mm}
 
Define $F(f) := \{\{\{a_1,b_1\},\{a_2, b_2\},\cdots , \{a_N , b_N\} \}\}$. Since $\{a_n : n \in \{1,2,\cdots n\}\} \cup \{b_n : n \in \{1,2,\cdots n\}\}$ is a set of $2N$ distinct elements in $\{1,2,\cdots 2N\}$, we get that $F(f)$ defines a \emph{matching} on $\{1,2,\cdots, 2N\}$. We claim that $F(f)$ represents a \emph{noncrossing matching on $2N$}. Towards a contradiction, assume there exists $i,j \in \{1,2,\cdots , N\}$ such that $a_i < a_j< b_i < b_j$. We will take two cases.
\begin{itemize}
    \item If $f(b_i)\leq f(b_j)$ then this is a contradiction to Property $(4)$ as $a_j<b_i<b_j$.
    \item If $f(b_i) > f(b_j)$ then Property (4) yields $f(a_i) = f(b_i) + 1 > f(b_j) + 1 = f(a_j)$, which is a contradiction to Property $(4)$ as $a_i < a_j < b_i$.
\end{itemize}

    \begin{tikzpicture}
        \draw[very thin,color=gray!15,step=1] (0,0) grid (6,1);
        \draw[-latex] (0,0) -- (6.2,0) node[right] { };
        \draw[-latex] (0,0) -- (0,1.2) node[above] {};
        \coordinate [label=below left:O] (a) at (0,0);
        \foreach \i in {1,2,...,6}
        \draw[gray!65] (\i,.1)--(\i,-.1) node[below] {$\i$};
        \foreach \i in {1}
        \draw[gray!65] (.1,\i)--(-.1,\i) node[left] {$\i$};
        \draw[thick] (0,0)--(1,1){};
        \draw[thick] (1,1)--(2,0){};
        \draw[thick] (2,0)--(3,1){};
        \draw[thick] (3,1)--(4,0){};
        \draw[thick] (4,0)--(5,1){};
        \draw[thick] (5,1)--(6,0){};
        \draw [-stealth](6.25,0.5) -- (8.2,0.5);
        \foreach \i in {1,2,...,6}
        \draw[black!65] (\i-1+8.5,.1)--(\i-1+8.5,-.1) node[below] {$\i$};
            \draw[very thin](0+8.5,0)--(5+8.5,0){};
            \draw[very thick] (1+8.5,0) arc (0:180:0.5);
            \draw[very thick] (3+8.5,0) arc (0:180:0.5);
            \draw[very thick] (5+8.5,0) arc (0:180:0.5);
\end{tikzpicture}\\
\begin{tikzpicture}
        \draw[very thin,color=gray!15,step=1] (0,0) grid (6,3);
        \draw[-latex] (0,0) -- (6.2,0) node[right] { };
        \draw[-latex] (0,0) -- (0,3.2) node[above] {};
        \coordinate [label=below left:O] (a) at (0,0);
        \foreach \i in {1,2,...,6}
        \draw[gray!65] (\i,.1)--(\i,-.1) node[below] {$\i$};
        \foreach \i in {1,2,3}
        \draw[gray!65] (.1,\i)--(-.1,\i) node[left] {$\i$};
        \draw[thick] (0,0)--(1,1){};
        \draw[thick] (1,1)--(2,2){};
        \draw[thick] (2,2)--(3,3){};
        \draw[thick] (3,3)--(4,2){};
        \draw[thick] (4,2)--(5,1){};
        \draw[thick] (5,1)--(6,0){};
        \draw [-stealth](6.25,1.5) -- (8.2,1.5);
        \foreach \i in {1,2,...,6}
        \draw[black!65] (\i-1+8.5,.1)--(\i-1+8.5,-.1) node[below] {$\i$};
            \draw[very thin](0+8.5,0)--(5+8.5,0){};
            \draw[very thick] (4+8.5,0) arc (0:180:1.5);
            \draw[very thick] (5+8.5,0) arc (0:180:2.5);
            \draw[very thick] (3+8.5,0) arc (0:180:0.5);
\end{tikzpicture}

    \begin{tikzpicture}
    
        \draw[very thin,color=gray!15,step=1] (0,0) grid (6,2);
        \draw[-latex] (0,0) -- (6.2,0) node[right] { };
        \draw[-latex] (0,0) -- (0,2.2) node[above] {};
        \coordinate [label=below left:O] (a) at (0,0);
        \foreach \i in {1,2,...,6}
        \draw[gray!65] (\i,.1)--(\i,-.1) node[below] {$\i$};
        \foreach \i in {1,2}
        \draw[gray!65] (.1,\i)--(-.1,\i) node[left] {$\i$};
        \draw[thick] (0,0)--(1,1){};
        \draw[thick] (1,1)--(2,0){};
        \draw[thick] (2,0)--(4,2){};
        \draw[thick] (4,2)--(5,1){};
        \draw[thick] (5,1)--(6,0){};
        \draw [-stealth](6.25,1) -- (8.2,1);
        \foreach \i in {1,2,...,6}
        \draw[black!65] (\i-1+8.5,.1)--(\i-1+8.5,-.1) node[below] {$\i$};
            \draw[very thin](8.5,0)--(13.5,0){};
            \draw[very thick] (9.5,0) arc (0:180:0.5);
            \draw[very thick] (13.5,0) arc (0:180:1.5);
            \draw[very thick] (12.5,0) arc (0:180:0.5);
    \end{tikzpicture}

    \begin{tikzpicture}
    
        \draw[very thin,color=gray!15,step=1] (0,0) grid (6,2);
        \draw[-latex] (0,0) -- (6.2,0) node[right] { };
        \draw[-latex] (0,0) -- (0,2.2) node[above] {};
        \coordinate [label=below left:O] (a) at (0,0);
        \foreach \i in {1,2,...,6}
        \draw[gray!65] (\i,.1)--(\i,-.1) node[below] {$\i$};
        \foreach \i in {1,2}
        \draw[gray!65] (.1,\i)--(-.1,\i) node[left] {$\i$};
        \draw[thick] (0,0)--(2,2){};
        \draw[thick] (2,2)--(4,0){};
        \draw[thick] (4,0)--(5,1){};
        \draw[thick] (5,1)--(6,0){};
        \draw [-stealth](6.25,1) -- (8.2,1);
        \foreach \i in {1,2,...,6}
        \draw[black!65] (\i-1+8.5,.1)--(\i-1+8.5,-.1) node[below] {$\i$};
            \draw[very thin](8.5,0)--(13.5,0){};
            \draw[very thick] (13.5,0) arc (0:180:0.5);
            \draw[very thick] (11.5,0) arc (0:180:1.5);
            \draw[very thick] (10.5,0) arc (0:180:0.5);
    \end{tikzpicture}

    \begin{tikzpicture}
        \draw[very thin,color=gray!15,step=1] (0,0) grid (6,2);
        \draw[-latex] (0,0) -- (6.2,0) node[right] { };
        \draw[-latex] (0,0) -- (0,2.2) node[above] {};
        \coordinate [label=below left:O] (a) at (0,0);
        \foreach \i in {1,2,...,6}
        \draw[gray!65] (\i,.1)--(\i,-.1) node[below] {$\i$};
        \foreach \i in {1,2}
        \draw[gray!65] (.1,\i)--(-.1,\i) node[left] {$\i$};
        \draw[thick] (0,0)--(1,1){};
        \draw[thick] (1,1)--(2,2){};
        \draw[thick] (2,2)--(3,1){};
        \draw[thick] (3,1)--(4,2){};
        \draw[thick] (4,2)--(5,1){};
        \draw[thick] (5,1)--(6,0){};
        \draw [-stealth](6.25,1.25) -- (8.2,1.25);
        \foreach \i in {1,2,...,6}
        \draw[black!65] (\i-1+8.5,.1)--(\i-1+8.5,-.1) node[below] {$\i$};
            \draw[very thin](0+8.5,0)--(5+8.5,0){};
            \draw[very thick] (2+8.5,0) arc (0:180:0.5);
            \draw[very thick] (5+8.5,0) arc (0:180:2.5);
            \draw[very thick] (4+8.5,0) arc (0:180:0.5);
\end{tikzpicture}

We will now show that $F$ is a bijective mapping from set of Dyck paths to the set of \emph{noncrossing matchings}.  For this we will define a map $G: \{\text{Non crossing matchings on }2N\} \to \{\text{Dyck paths to }(2N,0)\}$ such that $G \circ F = Id$ and $F \circ G = Id$. Fix any \emph{noncrossing matchings on $2N$}. Let $\{\{a_n,b_n\}: n \in \{1,2,\cdots, N\}\}$ represents that noncrossing matching. Without loss of generality, assume $a_n < a_{n+1}$ and $a_n < b_n$ for every $n$. Define a sequence of functions $(f_n : n \in \N )$ such that for every $n \in \{1,2,\cdots, N\}$, $\mathrm{Dom}(f_n) = \{0\}\cup \bigcup \limits_{k=1}^n \{a_k,b_k\}  $ and $\mathrm{Img}(f_n) \subseteq \N_0$.

\textbf{Base Case.} Define $f_1 : \{0,a_1, b_1\} \to \N_0$ such that $f(0) = f(b_1) = 0$ and $f(a_1) = 1$.

\vspace{-2mm}

\textbf{Induction Step.} Since $a_k < a_{k+1}<b_{k+1}$ for every $k \in \{1,2,\cdots, N-1\}$, we get that $\left(\bigcup \limits_{i=k}^n \{a_k,b_k\}\right) \cap \{1,2,\cdots, a_k-1\} = \emptyset$. Thus $\{ 1,2,\cdots, a_k-1\} \subseteq \left(\bigcup \limits_{i=1}^{k-1} \{a_k,b_k\}\right) $ for every $k \in \{1,2,\cdots, N\}$. Define $f_{n+1} : \{0\}\cup \bigcup \limits_{k=1}^{n+1} \{a_k,b_k\}  \to \N_0$ such that $f_{n+1}\mid_{\mathrm{Dom}(f_n)} = f_n$, $f_{n+1} (a_{n+1}) = f_{n+1} (a_{n+1}-1)+1$ and $f_{n+1} (b_{n+1}) = f(a_{n+1})-1$. 

\vspace{4mm}

By an inductive argument we can show that $f_N$ is a Dyck path. Define $G(\{\{a_n,b_n\}: n \in \{1,2,\cdots, N\}\}) = f_N$. Now by a trivial induction argument we can show that $F \circ G = Id$ and $G \circ F = Id$. Thus $F$ is a bijective. Hence $\Bar{t}(2N,0)$ equals the number of noncrossing  matchings on $2N$.
\end{proof}

\begin{cor}
    For any $N \in \N$, $t(0,2N)$ equals the number of noncrossing  matchings on $2N$.
\end{cor}
\bibliographystyle{plain}

\section{The Algebra $U_q (\SL)$}\label{Quantum group}
In this section, we will describe tilting modules corresponding to the restricted specialisation of quantised enveloping algebras. These tilting modules are central to the results proved in the following sections. Initially, we will recall the definition of a one-parameter deformation of the quantised universal  enveloping algebra of $\SL$, i.e. $\UU (\SL) $. This deformation also comes with a Hopf algebra structure. Fix a formal parameter $q $.

\begin{defn}\cite[\S~1.1]{jantzen1996lectures}
    $(\UU_q(\SL),+,\cdot)$ is defined as an associative algebra generated by $E,F,K$ and $K^{-1}$ with the following relations:
    \begin{align}
        KK^{-1} = &1 = K^{-1}K \tag{E1}\\
        KEK^{-1} &= q^2 E\tag{E2}\\
        KFK^{-1} &= q^{-2}F\tag{E3}\\
        EF-FE &= \frac{K-K^{-1}}{q - q^{-1}}\tag{E4}
    \end{align}
\end{defn}

\noindent{}As stated earlier in the section, $U_q(\SL)$ comes with a Hopf Algebra structure.

\begin{prop}\cite[\S~3]{jantzen1996lectures}
    There exists a unique algebra homomorphism $\Delta:\UU_q(\SL) \to \UU_q (\SL) \otimes \UU_q(\SL)$, $\varepsilon : \UU_q(\SL) \to \C$ and an anti-automorphism $S: \UU_q(\SL) \to \UU_q(\SL) $ such that, 
    \begin{gather*}
        \left(\UU_q (\SL), +, \cdot, \Delta, \varepsilon, S\right) \text{ is a Hopf Algebra; }\\
        \Delta(E) = K \otimes E + E \otimes 1 ,\, \Delta(F) = 1 \otimes F + F \otimes K^{-1},\, \Delta(K) = K \otimes K;\\
        \varepsilon(E) = \varepsilon(F) = 0 ,\, \varepsilon(K) = 1;\\
        S(E) = -K^{-1}E,\, S(F) = -FK, \, S(K) = K^{-1}.
    \end{gather*}
\end{prop}

We will now consider a case when $q$ is an odd root of unity. Assume $l$ to be the least natural number such that $q^l = 1$. For any $s \in \N$, define $$ [s]:= \frac{q^s-q^{-s}}{q-q^{-1}} \text{ and }[s]! = \prod_{i=1}^s [i] .$$  Chari and Presley\cite[\S~9.3]{chari1995guide} define a restricted specialisation $U_q^{\mathrm{res}}(\SL)$ of $U_q(\SL)$ as a subsalgebra generated by $\{K,K^{-1}\}\cup \left\{\frac{E^s}{[s]!}: s \in \N\right\} \cup \left\{\frac{F^r}{[r]!}: r \in \N\right\}$. They are also called the Lusztig's divided power subalgebra of $U_q(\SL).$ Define $W_q^{\mathrm{res}}(\lambda)$ \cite[\S~11.2]{chari1995guide} as a highest weight $U_q^{\mathrm{res}}(\SL)-$modules with highest weight $\lambda \in P$. A module $V$ is said to be type $\mathbf{1}$ module \cite[\S~4.6]{lusztig1989modular} if $K^l v = v$ for every $v \in V.$

\begin{defn}[Weyl filtration]\cite[\S~11.3.1]{chari1995guide}
    A finite-dimensional $U_q^{\mathrm{res}}(\SL)-$module $V$ of type $\mathbf{1}$ has a Weyl filtration if there exists a sequence of submodules $$ 0=V_0 \subset V_1 \subset V_2\subset \cdots \subset V_p = V$$ with $V_r / V_{r-1} \cong W_q^{\mathrm{res}}(\lambda_r)$ for some $\lambda_r \in P^+$ and $r \in \{1,2,\cdots, p\}$, where $P^+$ is the set of positive roots of the representation.
\end{defn}

After establishing the theoretical foundation we can finally give a definition of tilting modules.

\begin{defn}[Tilting modules]
    A tilting modules is a finite-dimensional $U^{\mathrm{res}}_q(\SL)-$module of type $\mathbf{1}$ such that both the module and its dual have Weyl filtrations.
\end{defn}
 
Chari and Presley \cite[\S~11.3.4]{chari1995guide} states that the for any given $k \in P^{+}$ there exists a unique indecomposable tilting module (up to isomorphism) of highest weight $k$. Define $T(k)$ as the unique indecomposable tilting module over $\UU_q^{\mathrm{res}} (\SL)$ of type $\mathbf{1}$ with maximum weight k. Similar to Clebsch-Gordan formulae in Proposition \ref{clebsch-gordan}, we also have an explicit formulae for direct sum decomposition of the tensor product of two indecomposable tilting module of type $\mathbf{1}$ as given below.

\begin{prop}\cite[\S~6]{anna2024multiplicities}\label{recurrence for tilting modules}
    For any $k_1 \in \N$ and $k_0 \in \{1,2,\cdots , l-3\}$, we have 
    \begin{align*}
        T(k_0 )\otimes T(1) &\cong T(k_0 + 1) \oplus T(k_0 -1)\\
        T(lk_1 + k_0) \otimes T(1) &\cong T(lk_1 + k_0 -1) \oplus T(lk_1 + k_0 + 1)\\
        T(lk_1 + 2l-2) \otimes T(1) &\cong T(lk_1 + 2l-3) \oplus T(lk_1 +2l -1 ) \oplus T(lk_1 - 1)\\
        T(2l-2) \otimes T(1) &\cong T(2l-1) \oplus T(2l - 3)\\
        T(l-2) \otimes T(1) &\cong T( l -1) \oplus T(l-3)\\
        T(lk_1 - 1)\otimes T(1) &\cong T(lk_1)\\
        T(lk_1)\otimes T(1) &\cong T(lk_1 + 1) \oplus \left(T(lk_1 - 1)\right)^{\oplus 2}
    \end{align*}
\end{prop}

\section{Direct sum decomposition of $T(1)^{\otimes N}$}\label{Section - p(k,N)}
\label{Bounded Catalan paths}
In a manner reminiscent of the case concerning $\SL$, Proposition \ref{recurrence for tilting modules} provides us with a valuable insight. It reveals that for any natural number $N$, there exists a direct sum decomposition of $T(1)^{\otimes N}$ into indecomposable tilting modules of type $\mathbf{1}$. Solovyev \cite{solovyev2022congruence} has notably devised a Lattice path model, incorporating filter restrictions, for effectively enumerating the multiplicities of modules within the decomposition of $T(1)^{\otimes N}$. In this section, we endeavor to shed light on an alternative combinatorial viewpoint regarding the multiplicity of $T(k)$ in $T(1)^{\otimes N}$ by harnessing the  concept of Catalan paths as defined in \ref{defn: Catalan paths}.

Let us set $T(1)^{\otimes 0} = T(0)$. Define a sequence $\left(p(k,N): k,N  \in \N_0\right)$ such that 
$$ T(1)^{\otimes N} \cong \bigoplus_{k \in \N} T(k)^{\oplus p(k,N)}, \text{ for every }N \in \N_0.$$

For any $N \in \N$, we fix $p(-1, N) := 0$.

\begin{thm}\cite{anna2024multiplicities}\label{explicit form for p}
    For any $k_1 \in \N_0$, $k_0 \in \{0,1,\cdots, l-1\}$ and $N \in \N_0$ we have 
       $$
        p(k_1l + k_0, N)  = \begin{cases}
            t(k_1 l + k_0,N) &\text{if }k_0 = l-1\\
            \sum\limits_{n \geq 0} t((k_1 + 2n)l + k_0, N) -\sum\limits_{m \geq 0} t((k_1 + 2m+ 2)l -  k_0 - 2, N) &\text{otherwise }
        \end{cases} $$
\end{thm}

\begin{cor}
\label{cor: recurrence for p(k,N)}
    For any $k_1 \in \N_0$ and $ N \in \N$  we have 
    \begin{itemize}
        \item For any $k_0 \in \{0,1,\cdots, l-3\}$, $ p(k_1l + k_0, N) = p(k_1l + k_0-1,N-1) + p(k_1l+k_0 + 1, N-1) $;
        \item For $ k_0 = l-2 $, $p(k_1 l + k_0, N) = p(k_1l + l-3,N-1)$.
    \end{itemize}
\end{cor}

\begin{proof}
    We will take three cases:
    \begin{itemize}
        \item If $k_0 = 0$, then \begin{align*}
            p(k_1 l + k_0, N) &= \sum\limits_{n \geq 0} t((k_1 + 2n)l , N) -\sum\limits_{m \geq 0} t((k_1 + 2m+ 2)l  - 2, N)\\
            &= \sum\limits_{n \geq 0} t((k_1 + 2n)l-1 , N-1) + t((k_1 + 2n)l+1 , N-1) \tag{by Proposition \ref{prop: recursion of t}}\\
            & \hspace*{14mm} -\sum\limits_{m \geq 0} t((k_1 + 2m+ 2)l  - 1, N-1) + t((k_1 + 2m+ 2)l  - 3, N-1)\\
            &= \sum\limits_{n \geq 0} t((k_1 + 2n)l-1 , N-1) - \sum\limits_{m \geq 1} t((k_1 + 2m)l  - 1, N-1)\\
            &\hspace*{16mm} +\sum\limits_{n \geq 0} t((k_1 + 2n)l+1 , N-1) - \sum\limits_{m \geq 1} t((k_1 + 2m)l  - 3, N-1)\\
            &= t(k_1l - 1, N-1) + p(k_1l+1, N-1)\tag{by Theorem \ref{explicit form for p}}\\
            &= p(k_1 l +k_0 -1,N-1) + p(k_1 l + k_0 +1, N-1).
        \end{align*}

        \item If $k_0 \in \{1,2,\cdots, l-3\},$ then \begin{align*}
            p(k_1 l + k_0, N) &= \sum\limits_{n \geq 0} t((k_1 + 2n)l + k_0 , N) -\sum\limits_{m \geq 0} t((k_1 + 2m+ 2)l -k_0 - 2, N)\\
            &= \sum\limits_{n \geq 0} t((k_1 + 2n)l+k_0-1 , N-1) + t((k_1 + 2n)l+k_0+1 , N-1) \tag{by Proposition \ref{prop: recursion of t}}\\
            & \hspace*{12mm} -\sum\limits_{m \geq 0} t((k_1 + 2m+ 2)l - k_0  - 1, N-1) + t((k_1 + 2m+ 2)l - k_0  - 3, N-1)\\
            &= \sum\limits_{n \geq 0} t((k_1 + 2n)l + k_0 -1 , N-1) - \sum\limits_{m \geq 1} t((k_1 + 2m)l -k_0 - 1, N-1)\\
            &\hspace*{14mm} +\sum\limits_{n \geq 0} t((k_1 + 2n)l+k_0+1 , N-1) - \sum\limits_{m \geq 1} t((k_1 + 2m)l - k_0  - 3, N-1)\\
            &= p(k_1 l +k_0 - 1, N-1) + p(k_1 l + k_0 + 1, N-1)\tag{by Theorem \ref{explicit form for p}}
        \end{align*}
        \item  If $k_0 = l-2,$ then \begin{align*}
             p(k_1 l + k_0, N) &= \sum\limits_{n \geq 0} t((k_1 + 2n)l + l-2 , N) -\sum\limits_{m \geq 0} t((k_1 + 2m+ 2)l -l, N)\\
            &= \sum\limits_{n \geq 0} t((k_1 + 2n+1)l-3 , N-1) + t((k_1 + 2n+1)l-1 , N-1) \tag{by Proposition \ref{prop: recursion of t}}\\
            & \hspace*{12mm} -\sum\limits_{m \geq 0} t((k_1 + 2m+ 1)l  - 1, N-1) + t((k_1 + 2m+ 1)l +1, N-1)\\
            &= \sum\limits_{n \geq 0}t((k_1 + 2n+1)l-1 , N-1) - \sum\limits_{m \geq 0} t((k_1 + 2m+ 1)l  - 1, N-1)\\
            &\hspace*{14mm} +\sum\limits_{n \geq 0} t((k_1 + 2n+1)l-3 , N-1) - \sum\limits_{m \geq 1} t((k_1 + 2m+1)l   +1, N-1)\\
            &= p(k_1 l + l-3, N-1)\tag{by Theorem \ref{explicit form for p}}
        \end{align*}
    \end{itemize}
\end{proof}

Similar to Section \ref{Section: catalan paths and t(k,N)}, we will now establish a one-one correspondence between $p(k,N)$ and some  class of Catalan paths ending at $(N,k)$. First, we will define a special type of Catalan path.

\begin{defn}
    For any $k_1,N \in \N_0$ and $k_0 \in \{-1, 0,\cdots , l-2\} $, a Catalan path $f$ ending at $(N,k_1 l + k_0) $ is said to be Partially $l-$bounded if for every $i \in f^{-1}(k_1l+l - 1)$ we have, $$\{i,i+1,\cdots, N\}\cap f^{-1}(k_1 l -1) \neq \emptyset.$$
\end{defn}

The next figure illustrates the definition of Partially $l-$bounded Catalan Paths.

   \begin{figure}[H]
\centering
\begin{subfloat} {
\label{ex:1}
\centering
    \begin{tikzpicture}
   
            \draw[very thin,color=gray!15,step=0.8] (0,0) grid (16.8,9.6);
        \draw[-latex] (0,0) -- (17,0) node[right] { };
        \draw[-latex] (0,0) -- (0,9.8) node[above] {};
        \coordinate [label=below left:O] (a) at (0,0);
        \foreach \i in {1,2,...,21}
        \draw[gray!65] (\i*0.8,.1)--(\i*0.8,-.1) node[below] {$\i$};
        \foreach \i in {1,2,...,12}
        \draw[gray!65] (.1,\i*0.8)--(-.1,\i*0.8) node[left] {$\i$};
        \draw[dashed] (0,3.2)--(16.8,3.2){};
        \draw[dashed] (0,7.2)--(16.8,7.2){};
        \draw[color=red!60,thin] (0,0)--(8,8){};
        \draw[color=red!60,thin] (8,8)--(13.6, 2.4){};
        \draw[color=red!60,thin] (13.6,2.4)--(16, 4.8){};
        \node at (8,8.3) {$(1)$};
        \draw[color=blue!60, thin] (0,0)--(0.8,0.8){};
        \draw[color=blue!60,thin] (0.8,0.8)--(1.6,0){};
        \draw[color=blue!60,thin] (1.6, 0)--(10.4, 8.8){};
        \node at (10.4,9.1) {$(2)$};
        \draw[color=blue!60,thin] (10.4,8.8)--(16,3.2){};
        \draw[color=green!60, thin] (0,0)--(1.6,1.6){};
        \draw[color=green!60, thin] (1.6,1.6)--(3.2,0){};
        \draw[color=green!60, thin] (3.2,0)--(5.6,2.4){};
        \node at (5.6,2.7) {$(3)$};
         \draw[color=green!60, thin] (5.6,2.4)--(8,0){};
         \draw[color=green!60, thin] (8,0)--(8.8,0.8){};
         \draw[color=green!60, thin] (8.8,0.8)--(9.6,0){};
         \draw[color=green!60, thin] (9.6,0)--(10.4,0.8){};
         \draw[color=green!60, thin] (10.4,0.8)--(11.2,0){};
         \draw[color=green!60, thin] (11.2,0)--(12,0.8){};
         \draw[color=green!60, thin] (12,0.8)--(12.8,0){};
         \draw[color=green!60, thin] (12.8,0)--(13.6,0.8){};\draw[color=green!60, thin] (13.6,0.8)--(14.4,0){};
         \draw[color=green!60, thin] (14.4,0)--(15.2,0.8){};
         \draw[color=green!60, thin] (15.2,0.8)--(16,0){};
         \draw[color=violet!60, thin] (0,0)--(4.8,4.8){};
         \draw[color=violet!60, thin] (4.8,4.8)--(6.4,3.2){};
         \draw[color=violet!60, thin] (6.4,3.2)--(12,8.8){};
         \draw[color=violet!60, thin] (12,8.8)--(16,4.8){};
         \node at (12,9.1) {$(4)$};
        \draw[color=purple!60, thin] (0,0)--(2.4,2.4){};
        \draw[color=purple!60, thin] (2.4,2.4)--(4.8,0){};
        \draw[color=purple!60, thin] (4.8,0)--(8,3.2){};
        \draw[color=purple!60, thin] (8,3.2)--(10.4,0.8){};
        \draw[color=purple!60, thin] (10.4,0.8)--(12,2.4){};
        \draw[color=purple!60, thin] (12,2.4)--(13.6,0.8){};
        \draw[color=purple!60, thin] (13.6,0.8)--(15.2,2.4){};
        \draw[color=purple!60, thin] (15.2,2.4)--(16,1.6){};
        \node at (8,3.5) {$(5)$};
    \end{tikzpicture}
}
\end{subfloat}
\caption*{\mbox{$(1),(2)$ and $(3)$ are Partially 5-bounded Catalan paths while $(4)$ and $(5)$ are not Partially 5-bounded.}}
\label{Example:1}

\end{figure}

For $N \in \N_0, k_1 \in \N_0$ and $k_0 \in \{0,1,\cdots, l-1\},$ we define $A^N_{k_1, k_0}$ as the set of all Partially $l-$bounded catalan paths ending at $(N,k_1l+k_0-1)$. Here, we fix $A^N_{0,0}:= \emptyset$ for every $N \in \N_0.$ Since $l>1$, we can note one important observation.

\begin{rem}\label{rem:l-1barrier}
    For any $k_1 \in  \N_0, N\in \N$ and $k_0 \in \{0,1,\cdots , l-1\}$, if $f \in A^N_{k_1, k_0}$ then $f(N-1) \neq k_1l + l-1.$
\end{rem}

If $k_1 = 0$, then note that $f^{-1}(k_1l - 1) = \emptyset$ for any catalan path $f$. Thus if $f \in A^N_{0, k_0}$ then $f^{-1}(l-1) = \emptyset.$ Hence we can note an important observation.

\begin{rem}\label{rem: bounded catalan paths}
    For any $N \in \N$ and $k_0 \in \{ 0 , 1, \cdots , l-1\}$, $$A^N_{0, k_0} = \{ f : f\text{ is a Catalan path ending at }(N, k_0 - 1) \text{ and } f^{-1}(k_1l + l - 1) = \emptyset\}.$$
\end{rem}

The problem of enumeration of paths in Remark \ref{rem: bounded catalan paths} have been studied in \cite{bousquet2007discrete} and \cite{bacher2013generalized}.

\begin{prop}
\label{prop: recurrence of special catalan}
    For any $N \in \N,k_1 \in \N_0$ and $k_0 \in \{0,1,\cdots, l-1\}$, we have 
    $$\left|A^N_{k_1, k_0}\right| = \begin{cases}
        \left|A^{N-1}_{k_1,k_0-1}\right|+\left|A^{N-1}_{k_1,k_0+1}\right| &\text{if }0<k_0<l-1

\vspace{1.5mm}\\
        
        \left|A^{N-1}_{k_1,k_0-1}\right|&\text{if }k_0 = l-1

\vspace{1.5mm}\\
        
        t(k_1 l + k_0 -1, N)&\text{if }k_0 = 0
    \end{cases}.$$
\end{prop}

\begin{proof}
   Fix any catalan path $f$ ending at $(N, k_1 l + k_0 - 1)$. Thus $f(N-1) \in \{k_1l+k_0-2, k_1l + k_0\}$  We have three cases. \begin{itemize}
        \item If $k_0 \in \{ 1,2\cdots , l-2 \}$, then $0\leq k_0-1<k_0 + 1\leq l-1$. Since $N \notin f^{-1}(k_1 l-1) $ we get that $\{i,i+1, \cdots, N-1\} \cap f^{-1}(k_1 l -1) = \{i,i+1, \cdots, N\} \cap f^{-1}(k_1 l - 1) $ for any $i \in f^{-1} (k_1 l + l-1)$. Hence $ f\mid_{ \{0,1,\cdots , N-1\} } \in A^N_{k_1, k_0 - 1} \cup A^N_{k_1 , k_0 + 1}$ if and only if $f \in A^N_{k_1 , k_0}$. Therefore, $$ \left\lvert A^N_{k_1 , k_0} \right\rvert = \ = \left\lvert A^{N-1}_{k_1,k_0 - 1} \right\rvert + \left\lvert A^{N-1}_{k_1, k_0 + 1} \right\rvert .$$

        \item If $k_0 = l-1$, then Remark \ref{rem:l-1barrier} yields that if $f \in A^N_{k_1,k_0}$ then $f(N-1) \neq k_1l + k_0-1 = k_1l+k_0$. Thus is $f \in A^N_{k_1,k_0}$ then $f(N-1) = k_1 l + k_0 -2$. Since $N \notin f^{-1}(k_1 l-1) $ we get that $\{i,i+1, \cdots, N-1\} \cap f^{-1}(k_1 l -1) = \{i,i+1, \cdots, N\} \cap f^{-1}(k_1 l - 1) $ for any $i \in f^{-1} (k_1 l + l-1)$. Hence $ f\mid_{ \{0,1,\cdots , N-1\} } \in A^N_{k_1, k_0 - 1}$ if and only if $f \in A^N_{k_1 , k_0}$. Therefore, $$ \left\lvert A^N_{k_1 , k_0} \right\rvert =  \left\lvert A^{N-1}_{k_1,k_0 - 1} \right\rvert .$$

        \item If $k_0 = 0,$ then since $f(N) = k_1 l + k_0 -1 k_1l-1$ we get that $N \in \{ i, i+1 , \cdots , N \}\cap f^{-1} \neq \emptyset$ for every $i \in \{0,1,\cdots, N\}$. Hence $f \in A^N_{k_1, k_0}$ for any catalan path $f$ ending at $(N, k_1l + k_0 - 1)$. Therefore, $$ \left\lvert A^N_{k_1 , k_0} \right\rvert = t(k_1l + k_0 -1,N). $$
    \end{itemize}
\end{proof}

Using the recurrence relations obtained in Proposition \ref{prop: recurrence of special catalan} and Corollary \ref{cor: recurrence for p(k,N)}, we can establish a relation between $p(k,N)$ and $\left|A^N_{k_1, k_0}\right|$.

\begin{cor}\label{cor: p(k,N) for catalan paths}
    For any $k_1 , N\in \N_0$ and $k_0 \in \{0,1,\cdots, l-1\}$, $p(k_1 l + k_0 - 1, N) = \left\lvert A^N_{k_1 , k_0}\right\rvert$.
\end{cor}
\begin{proof}
    We will prove this by induction on $N \in \N$.

\vspace{5mm}

\textbf{Base Case.} If $N =0$, then note that $t(-1,0) = 0$ and $t(k,0) = 0$ if $k > 0$. Thus, Theorem \ref{explicit form for p} yields that 
\begin{align*}
    p(k_1 l +k_0-1,0) &= \begin{cases}
    t(k_1 l+ l -1, 0) &\text{if }k_0 = 0\\
    t(k_1 l + k_0-1 , 0) - t((k_1+2)l - k_0 -2, 0)&\text{otherwise}
\end{cases}\\
\implies p(k_1 l +k_0-1,0) &= \begin{cases}
    1 &\text{if }k_0 = 1\text{ and }k_0 = 0\\
    0&\text{otherwise}
\end{cases}
\end{align*}
 we have two cases.
\begin{itemize}
    \item If $k_1l+k_0 = 1$, then note that there exists a unique namely $f:\{ 0 \} \in \N_0$ where $f(0) = 0$. Thus $|A^N_{k_1,k_0}| = 1 = p(k_1l+k_0 - 1, N)$. 
    \item If $k_1l + k_0 \neq 1$, then note that $A^N_{k_1,k_0} = \emptyset$. Thus $\left|A^N_{k_1,k_0}\right| =0 = p(k_1l + k_0 - 1, N)$. 
\end{itemize}

\vspace{2mm}

    \textbf{Induction Step.} If $N > 0$, then Theorem \ref{explicit form for p} and Corollary \ref{cor: recurrence for p(k,N)} yields 
    $$p(k_1 l + k_0-1,N) = \begin{cases}
        p(k_1 l + k_0 -2, N-1)+p(k_1 l + k_0 , N-1) &\text{if }0<k_0<l-1\\
        p(k_1 l + k_0 -2, N-1)&\text{if }k_0 = l-1\\
        t(k_1 l + k_0 -1, N)&\text{if }k_0 = 0
    \end{cases}.$$
    By induction hypothesis we have 
    $$p(k_1 l + k_0-1,N) = \begin{cases}
        \left|A^{N-1}_{k_1,k_0-1}\right|+\left|A^{N-1}_{k_1,k_0+1}\right| &\text{if }0<k_0<l-1

        \vspace{1.5mm}\\
        
        \left|A^{N-1}_{k_1,k_0-1}\right|&\text{if }k_0 = l-1        
        \vspace{1.5mm}\\
        t(k_1 l + k_0 -1, N)&\text{if }k_0 = 0
    \end{cases}.  $$
    Thus Corollary \ref{prop: recurrence of special catalan} yields $ p(k_1 l + k_0 - 1, N) = \left| A^N_{k_1,k_0} \right|. $
\end{proof}

Corollary \ref{cor: finite sum for t terms}, also leads to another recursive description for $p(k,N)$ as given in the Proposition below.

\begin{prop}
    If $k_1 , N \in \N, k_0 \in \{0,1,\cdots ,l-1\}$ and $k' \in \{ 1,2,\cdots, k_1 \}$, then $$ p(k_1 l + k_0 , N) = \sum_{i=0}^{N-1} t(k'l-1, N-i-1)p((k_1 - k')l + k_0 , i). $$
\end{prop}

\begin{proof}
    We will take two cases.
        \begin{enumerate}
            \item If $k_0 = l-1,$ then Theorem \ref{explicit form for p} yields $p(k_1 l + k_0, N) = t(k_1 l + k_0 , N)  $. Thus Corollary \ref{cor: finite sum for t terms} yields $p(k_1 l + k_0, N) = \sum_{i=0}^{N-1} t(k'l-1, N-i-1)t((k_1-k')l + k_0 , i)  $. Since $k_0 = l-1$, we get that $$ p(k_1l+k_0, N) = \sum_{i=0}^{N-1} t(k'l-1, N-i-1)p((k_1-k')l + k_0 , i) . $$

            \item If $k_0 \in \{0,1,\cdots , l-2\},$ then
            \begin{align*}
               p(k_1 l + k_0) &= \sum\limits_{n \geq 0} t((k_1 + 2n)l + k_0, N) -\sum\limits_{m \geq 0} t((k_1 + 2m+ 2)l -  k_0 - 2, N)\\
                &= \sum\limits_{n \geq 0} \sum_{i=0}^{N-1} t(k'l-1,N-i-1)t((k_1 + 2n - k')l + k_0, i) - \tag{by Corollary \ref{cor: finite sum for t terms}}\\ &\hspace{30mm}\sum\limits_{m \geq 0} \sum_{i=0}^{N-1} t(k'l-1,N-i-1)t((k_1 + 2m + 2 - k')l - k_0 - 2, i) \\
                &= \sum_{i=0}^{N-1} t(k'l-1, N-i-1)\bigg( \sum_{n\geq 0} t((k_1 - k' + 2n )l + k_0, i) - \\&\hspace*{65mm}\sum_{m \geq 0} t((k_1 - k' + 2m + 2)l - k_0 - 2, i) \bigg)\\
                &= \sum_{i=0}^{N-1} t(k'l-1,N-i-1)p((k_1 - k')l + k_0, i)\tag{by Theorem \ref{explicit form for p}}
            \end{align*}
        \end{enumerate}
\end{proof}

\begin{cor}
    For any $N, k_1 \in \N,$ and $ k_0 \in \{0,1,\cdots , l-1\}$,
    $$ p(k_1 l + k_0, N) = \sum_{i=0}^{N-1} t(k_1 l - 1, i)p(k_0 ,N-i-1 ). $$
\end{cor}

The significance of the last Corollary lies in its ability to provide a comprehensive description of $p(k,N)$ in terms of $t$-terms and $p(k_0, i)$-terms, where $k_0 \in \{0,1,\cdots, l-1\}$. This result, in conjunction with Corollary \ref{cor: p(k,N) for catalan paths} and Remark \ref{rem: bounded catalan paths}, provides a connection to the enumeration of Catalan paths with restricted height.

The enumeration of such paths has been extensively explored by Bosquet-Melou\cite{bousquet2007discrete} and Bacher\cite{bacher2013generalized}. They have derived a generating function capable of enumerating Catalan paths with a bounded height, ending at a fixed point. Furthermore, they have also provided combinatorial interpretations for these functions by employing a transfer matrix method.

\section{Asymptotic behavior of $p(k,N)$}
\label{Asymptotic behavior}

As the name suggests, in this section we will investigate the asymptotic behaviour of the function as $N \to \infty$. We will focus on the case when $k < l-1$. Corollary \ref{cor: p(k,N) for catalan paths} and Remark \ref{rem: bounded catalan paths} yields that $p(k,N)$ is the number of Catalan paths to $(N,k)$ which do not touch the line $y = l-1$ line. We can use this combinatorial intuition to formulate a new method of finding $p(k,N)-$terms using a \emph{tridiagonal Toeplitz Matrix} \cite{noschese2013tridiagonal}. Following the notations defined in \cite[\S~1]{noschese2013tridiagonal}, define an $(l-1)$-dimensional tridiagonal Toeplitz Matrix $T := (l-1 ; 1,0,1),$ i.e.
$$ T:= \begin{pmatrix}
    0 & 1 & & & & &\\
    1 & 0 & 1 & & & &\\
    & 1 & 0 & 1 & & &\\
    & & \ddots & \ddots & \ddots & &\\
    & & & 1 & 0 & 1 &\\
    & & & & 1 & 0 & 1\\
    & & & & & 1 & 0
\end{pmatrix}. $$ According to the results in Section \ref{Section - p(k,N)}, we can note an important observation.

\begin{prop}\label{prop Relation between p(k,N) and T}
    For any $N \in \N_0$, $T^{N} e_1 = \sum_{k=1}^{l-1} p(k-1, N)e_k$.
\end{prop}
\begin{proof}
    We will prove the result by induction on $N.$

    \textbf{Base Case.} \emph{Trivial}

    \textbf{Induction Step.} Induction hypothesis yields that, $T^{N-1} e_1 = \sum_{k=1}^{l-1} p(k-1, N-1)e_k$. Hence,

    \begin{align*}
        T^{N} e_1 = \sum_{k=1}^{l-1} p(k-1, N-1)Te_k &= p(0,N-1)e_2 + \left(\sum_{k=2}^{l-2} p(k-1, N-1)  (e_{k-1} + e_{k+1})\right) + p(l-2,N-1) e_{l-2}\\
        &= p(1,N-1)e_1 + \left(\sum_{k=2}^{l-2} (p(k-2, N-1) + p(k,N-1))  e_k\right) + p(l-3, N-1)e_{l-1}\\
        &= \sum_{k=1}^{l-1} p(k-1, N)e_k\tag{by Corollary \ref{cor: recurrence for p(k,N)}}
    \end{align*}
\end{proof}

\begin{cor}\label{cor: Relation between p(k,N) and T}
    For any $N \in \N$ and $k \in \{0,1,\cdots, l-2\}$, we have $p(k,N) = \langle T^{N}e_1, e_{k+1} \rangle.$
\end{cor}

We will now use these results to formulate a method to find $p(k,N)$ for $k \in \{0,1,\cdots l-2\}$. For that, we will have to note some essential properties of a tridiagonal Toeplitz Matrix.

\begin{lem}\cite[\S~2]{noschese2013tridiagonal}\label{lem: eigenvectors and eigenvalues of T}
    For any $p \in \{1,2,\cdots, l-1\}$, $\lambda_s := 2\cos\left(\frac{s\pi}{l}\right)$ and $v_s := \sum_{q=1}^{l-1} \sin\left( \frac{sq\pi }{l} \right) e_q$ is the eigenvalue and the corresponding eigenvector of $T$.
\end{lem}

Also note that for any $q \in \{1,2,\cdots, l-1\},$ we get 
\begin{align*}
    \sum_{s=1}^{l-1} \sin\left( \frac{sq\pi }{l} \right)v_s = \sum_{s=1}^{l-1} \sum_{r=1}^{l-1} \sin\left( \frac{sq\pi }{l} \right)\sin\left( \frac{sr\pi }{l} \right)e_r = \sum_{r=1}^{l-1} \sum_{s=1}^{l-1} \sin\left( \frac{sq\pi }{l} \right)\sin\left( \frac{sr\pi }{l} \right)e_k
 \end{align*}
 \begin{align}\label{Equation: trigonometric form of eigenvectors}
     \implies \sum_{s=1}^{l-1} \sin\left( \frac{sq\pi }{l} \right)v_i = \frac{1}{2}\sum_{r=1}^{l-1} \sum_{s=1}^{l-1} \left( \cos\left(\frac{s (q-r)\pi}{l}\right) - \cos\left(\frac{s(q+r)\pi}{l}\right) \right)e_r
 \end{align}
Note that for any $n \in \N$, if $ \sin\left(\frac{n\pi}{2l}\right) \neq 0 $ then 
\begin{align}
    \sum_{s=1}^{l-1} \cos\left(\frac{sn\pi}{l}\right) &= \frac{1}{2\sin\left(\frac{n\pi}{2l}\right)} \sum_{s=1}^{l-1} \sin\left(\left(s + \frac{1}{2}\right)\frac{n\pi}{l}\right) - \sin\left(\left(s - \frac{1}{2}\right)\frac{n\pi}{l}\right)\notag \\
    &= \frac{1}{2\sin\left(\frac{n\pi}{2l}\right)} \left(\sin\left(\left(l - \frac{1}{2}\right)\frac{n\pi}{l}\right) - \sin\left( \frac{n\pi}{2l}\right)\right) \notag\\
    &= \frac{1}{2\sin\left(\frac{n\pi}{2l}\right)} \left( (-1)^{n+1} \sin\left(\frac{n\pi}{2l}\right) - \sin\left( \frac{n\pi}{2l}\right)\right)\notag.\\
    \implies \sum_{s=1}^{l-1} \cos\left(\frac{sn\pi}{l}\right) &= \frac{(-1)^{n+1} - 1}{2} \text{, whenever } \sin\left(\frac{n\pi}{2l}\right) \neq 0.\label{eq: sum of cosines}
\end{align}

Thus for any distinct $q,r \in \{1,2,\cdots, l-1\}$, we have $ \sum_{s=1}^{l-1} \cos\left(\frac{s (q-r)\pi}{l}\right) = \frac{(-1)^{q-r}-1}{2} = \frac{(-1)^{q+r - 1}}{2} = \sum_{s=1}^{l-1} \cos\left(\frac{s (q-r)\pi}{l}\right)$ and moreover $\sum_{s=1}^{l-1} \left(\cos\left(\frac{s (q-r)\pi}{l}\right) - \cos\left(\frac{s (q+r)\pi}{l}\right) \right) = 0$. Hence Equation \ref{Equation: trigonometric form of eigenvectors} and \ref{eq: sum of cosines} yields that,

$$ \sum_{s=1}^{l-1} \sin\left( \frac{sq\pi }{l} \right)v_s  = \frac{1}{2} \sum_{s =1}^{l-1} \left( 1 - \cos\left( \frac{2sq}{l} \right) \right) e_q = \frac{l-1}{2} + \frac{1}{2}$$

\begin{equation}\label{Relation between eigenvectors and standard basis}
    \implies \sum_{s=1}^{l-1} \sin\left( \frac{sq\pi }{l} \right)v_s = \frac{l}{2} e_q, \text{ for any }q\in \{1,2,\cdots, l-1\}
\end{equation}

\begin{prop}\label{prop: p(k,N) in terms of sine and cosines}
    For any $N \in \N$ and $k \in \{0,1,\cdots, l-2\}$, $$ p(k,N) = \frac{2^{N+1}}{l}\sum_{s=1}^{l-1} \cos^N\left( \frac{s\pi}{l}\right)\sin\left(\frac{s\pi}{l} \right)\sin\left(\frac{s(k+1)\pi}{l}\right).  $$
\end{prop}
\begin{proof}
    Corollary \ref{cor: Relation between p(k,N) and T} yields that 
    \begin{align*}
        p(k,N) &= \langle T^Ne_1 , e_{k+1} \rangle\\
        &= \left\langle \frac{2}{l}  \sum_{s=1}^{l-1} \sin\left( \frac{s\pi }{l} \right)T^N  v_s ,  \frac{2}{l}\sum_{i=1}^{l-1} \sin\left( \frac{s(k+1)\pi }{l} \right)v_s\right\rangle\tag{by Equation \ref{Relation between eigenvectors and standard basis}}\\
        &= \frac{4}{l^2} \left\langle \sum_{s=1}^{l-1} 2^N\cos^N\left(\frac{s\pi}{l}\right)\sin\left( \frac{s\pi }{l} \right)v_s , \sum_{s=1}^{l-1}\sin\left( \frac{s(k+1)\pi }{l} \right)v_s \right\rangle \tag{by Lemma \ref{lem: eigenvectors and eigenvalues of T}}
        \\
    &= \frac{4}{l^2} \sum_{s=1}^{l-1}\sum_{q=1}^{l-1} 2^N\cos^N\left(\frac{s\pi}{l}\right)\sin\left( \frac{s\pi }{l} \right)\sin\left( \frac{q(k+1)\pi }{l} \right)\langle v_s, v_q \rangle
    \end{align*}
    Since eigenvectors with different eigenvalues are orthogonal, we get that 
    \begin{align*}
        p(k,N) &= \frac{4}{l^2}  \sum_{s=1}^{l-1} 2^N\cos^N\left(\frac{s\pi}{l}\right)\sin\left( \frac{s\pi }{l} \right)\sin\left( \frac{s(k+1)\pi }{l} \right)\langle v_s, v_s \rangle
    \end{align*}
    For any $s \in \{1,2, \cdots, l-1\}$, Equation \ref{eq: sum of cosines} yields
    $$ \langle v_s, v_s\rangle = \sum_{q=1}^{l-1} \sin^2 \left(\frac{sq\pi}{l}\right) = \frac{1}{2}\sum_{q=1}^{l-1} \left(1- \cos \left(\frac{2sq\pi}{l}\right)\right)  = \frac{l}{2}.$$
    Therefore,
    $$ p(k,N) = \frac{2}{l}  \sum_{s=1}^{l-1} 2^N\cos^N\left(\frac{s\pi}{l}\right)\sin\left( \frac{s\pi }{l} \right)\sin\left( \frac{s(k+1)\pi }{l} \right) .$$
\end{proof}

The last proposition helps us to examine the asymptotic behaviour of $p(k,N).$ We will now state the asymptotic results which we can derive from Proposition \ref{prop: p(k,N) in terms of sine and cosines}.

\begin{cor}
    For any $N \in \N, k \in \{ 0,1,\cdots , l-2 \}$, we have $p(k,N) \leq 2^{N} \cos^N\left( \frac{\pi}{l} \right)$.
\end{cor}

\begin{proof}
    Since $\cos\left( \frac{\pi}{l} \right) \geq \left| \cos\left( \frac{s\pi}{l} \right) \right|$ for any $s \in \{1,2,\cdots, l-1\},$ Proposition \ref{prop: p(k,N) in terms of sine and cosines} yields
    \begin{align*}
        p(k,N) &\leq \frac{2}{l}  \sum_{s=1}^{l-1} \left|2^N\cos^N\left(\frac{\pi}{l}\right)\sin\left( \frac{s\pi }{l} \right)\sin\left( \frac{s(k+1)\pi }{l} \right)\right| \\
        \implies p(k,N) &\leq \frac{2^{N+1} \cos^N\left(\frac{\pi}{l}\right)}{l}  \sum_{s=1}^{l-1} \left|\sin\left( \frac{s\pi }{l} \right)\sin\left( \frac{s(k+1)\pi }{l} \right)\right|\\
        \implies p(k,N) &\leq \frac{2^{N+1} \cos^N\left(\frac{\pi}{l}\right)}{l}  \left(\sum_{s=1}^{l-1} \sin^2\left( \frac{s\pi }{l} \right)\right)^{\frac{1}{2}}\left(\sum_{s=1}^{l-1} \sin^2\left( \frac{s(k+1)\pi }{l} \right)\right)^{\frac{1}{2}} \tag{by Cauchy-Schwarz inequality}\\
         \implies p(k,N) &\leq \frac{2^{N}}{l} \cos^N\left( \frac{\pi}{l} \right) \left[\sum_{s=1}^{l-1} \left( 1 - \cos\left( \frac{2s\pi }{l}\right) \right)\right]^{\frac{1}{2}}\left[\sum_{s=1}^{l-1}\left( 1- \cos\left( \frac{2s(k+1)\pi }{l} \right)\right)\right]^{\frac{1}{2}}\\
        \implies p(k,N) &\leq 2^{N} \cos^N\left( \frac{\pi}{l} \right)\tag{by Equation \ref{eq: sum of cosines} }
    \end{align*}
\end{proof}

\begin{cor}
    For any $N \in \N$ and $k \in \{0,1,\cdots, l-2\},$ $$ \left\lvert p(k,N) - \frac{1 + (-1)^{k+N}}{l}\cdot 2^{N+1} \cos^N\left( \frac{\pi}{l} \right)\sin\left( \frac{\pi}{l} \right)\sin\left( \frac{\pi (k+1)}{l} \right) \right\rvert \leq 2^N\left|\cos^N\left( \frac{2\pi}{l} \right)\right| .$$ Moreover $$p(k,N) =  \frac{1 + (-1)^{k+N}}{l}\cdot 2^{N+1} \cos^N\left( \frac{\pi}{l} \right)\sin\left( \frac{\pi}{l} \right)\sin\left( \frac{\pi (k+1)}{l} \right) + O\left(2^N \left|\cos^N \left( \frac{2\pi}{l} \right)\right|\right).$$
\end{cor}

\begin{proof}
   Recall that in Section \ref{Quantum group}, $q$ is assumed to be an odd root of unity. Thus $l \neq 2.$ If $l = 3$, then Proposition \ref{prop: p(k,N) in terms of sine and cosines} yields 
   \begin{align*}
        p(k,N) &=  \frac{ 2^{N+1}}{3}\cdot \cos^N\left( \frac{\pi}{3} \right)\sin\left( \frac{\pi}{3} \right)\sin\left( \frac{\pi (k+1)}{3} \right) + \frac{ 2^{N+1}}{3}\cdot \cos^N\left( \frac{2\pi}{3} \right)\sin\left( \frac{2\pi}{3} \right)\sin\left( \frac{2\pi (k+1)}{3} \right)\\ 
        &= \frac{1 + (-1)^{k+N}}{3}\cdot 2^{N+1} \cos^N\left( \frac{\pi}{3} \right)\sin\left( \frac{\pi}{3} \right)\sin\left( \frac{\pi (k+1)}{3} \right). \\
   \end{align*}
   Hence the Corollary holds. Now assume $l>3$. Thus Proposition \ref{prop: p(k,N) in terms of sine and cosines} yields \begin{align*}
       p(k,N) &= \frac{2^{N+1}}{l}\sum_{s=1}^{l-1} \cos^N\left( \frac{s\pi}{l}\right)\sin\left(\frac{s\pi}{l} \right)\sin\left(\frac{s(k+1)\pi}{l}\right)\\
       &= \frac{2^{N+1}}{l}\cdot \left( 1+ (-1)^{k+N} \right)\cos^N\left( \frac{\pi}{l} \right)\sin\left( \frac{\pi}{l} \right)\sin\left( \frac{\pi (k+1)}{l} \right) + \frac{2^{N+1}}{l}\sum_{s=2}^{l-2} \cos^N\left( \frac{s\pi}{l}\right)\sin\left(\frac{s\pi}{l} \right)\sin\left(\frac{s(k+1)\pi}{l}\right)
   \end{align*}

   $$\hspace*{-8mm} \implies \left| p(k,N)- \frac{\left( 1+ (-1)^{k+N} \right)}{l}\cdot 2^{N+1}\cos^N\left( \frac{\pi}{l} \right)\sin\left( \frac{\pi}{l} \right)\sin\left( \frac{\pi (k+1)}{l} \right) \right| \leq \frac{2^{N+1}}{l}\sum_{s=2}^{l-2} \left|\cos^N\left( \frac{s\pi}{l}\right)\sin\left(\frac{s\pi}{l} \right)\sin\left(\frac{s(k+1)\pi}{l}\right)\right|$$
   Note that $ \left| \cos\left( \frac{2\pi}{l} \right) \right| \geq \left| \cos\left( \frac{s\pi}{l} \right) \right| $ for any $s \in \{2,3\cdots, l-2\}$. Hence,
   \begin{equation}
   \left| p(k,N)- \frac{\left( 1+ (-1)^{k+N} \right)}{l}\cdot 2^{N+1}\cos^N\left( \frac{\pi}{l} \right)\sin\left( \frac{\pi}{l} \right)\sin\left( \frac{\pi (k+1)}{l} \right) \right| \leq \frac{2^{N+1}}{l} \left|\cos^N\left( \frac{2\pi}{l}\right)\right| \sum_{s=2}^{l-2} \left|\sin\left(\frac{s\pi}{l} \right)\sin\left(\frac{s(k+1)\pi}{l}\right)\right|. \label{eq: final corollary}
   \end{equation}
Thus Cauchy-Schwarz inequality yields,
\begin{align*}
    \sum_{s=2}^{l-2} \left|\sin\left(\frac{s\pi}{l} \right)\sin\left(\frac{s(k+1)\pi}{l}\right)\right| &\leq \left(\sum_{s=2}^{l-2} \sin^2\left( \frac{s\pi}{l} \right) \right)^{\frac{1}{2}} \left(\sum_{s=2}^{l-2} \sin^2\left( \frac{s(k+1)\pi}{l} \right) \right)^{\frac{1}{2}}\\
    \implies \sum_{s=2}^{l-2} \left|\sin\left(\frac{s\pi}{l} \right)\sin\left(\frac{s(k+1)\pi}{l}\right)\right| &\leq \left(\sum_{s=1}^{l-1} \sin^2\left( \frac{s\pi}{l} \right) \right)^{\frac{1}{2}} \left(\sum_{s=1}^{l-1} \sin^2\left( \frac{s(k+1)\pi}{l} \right) \right)^{\frac{1}{2}}\\
    \implies \sum_{s=2}^{l-2} \left|\sin\left(\frac{s\pi}{l} \right)\sin\left(\frac{s(k+1)\pi}{l}\right)\right| &\leq \frac{1}{2} \left(\sum_{s=1}^{l-1} 1-\cos\left( \frac{2s\pi}{l} \right) \right)^{\frac{1}{2}} \left(\sum_{s=1}^{l-1} 1-\cos\left( \frac{2s(k+1)\pi}{l} \right) \right)^{\frac{1}{2}}\\
    \implies \sum_{s=2}^{l-2} \left|\sin\left(\frac{s\pi}{l} \right)\sin\left(\frac{s(k+1)\pi}{l}\right)\right| &\leq \frac{l}{2}\tag{by Equation \ref{eq: sum of cosines}}
\end{align*}
Therefore Equation \ref{eq: final corollary} yields 
$$\left\lvert p(k,N) - \frac{1 + (-1)^{k+N}}{l}\cdot 2^{N+1} \cos^N\left( \frac{\pi}{l} \right)\sin\left( \frac{\pi}{l} \right)\sin\left( \frac{\pi (k+1)}{l} \right) \right\rvert \leq 2^N\left|\cos^N\left( \frac{2\pi}{l} \right)\right|.$$
\end{proof}

We can see this asymptotic behaviour in the following plot.

\vspace{5mm}

\includegraphics[scale=0.5]{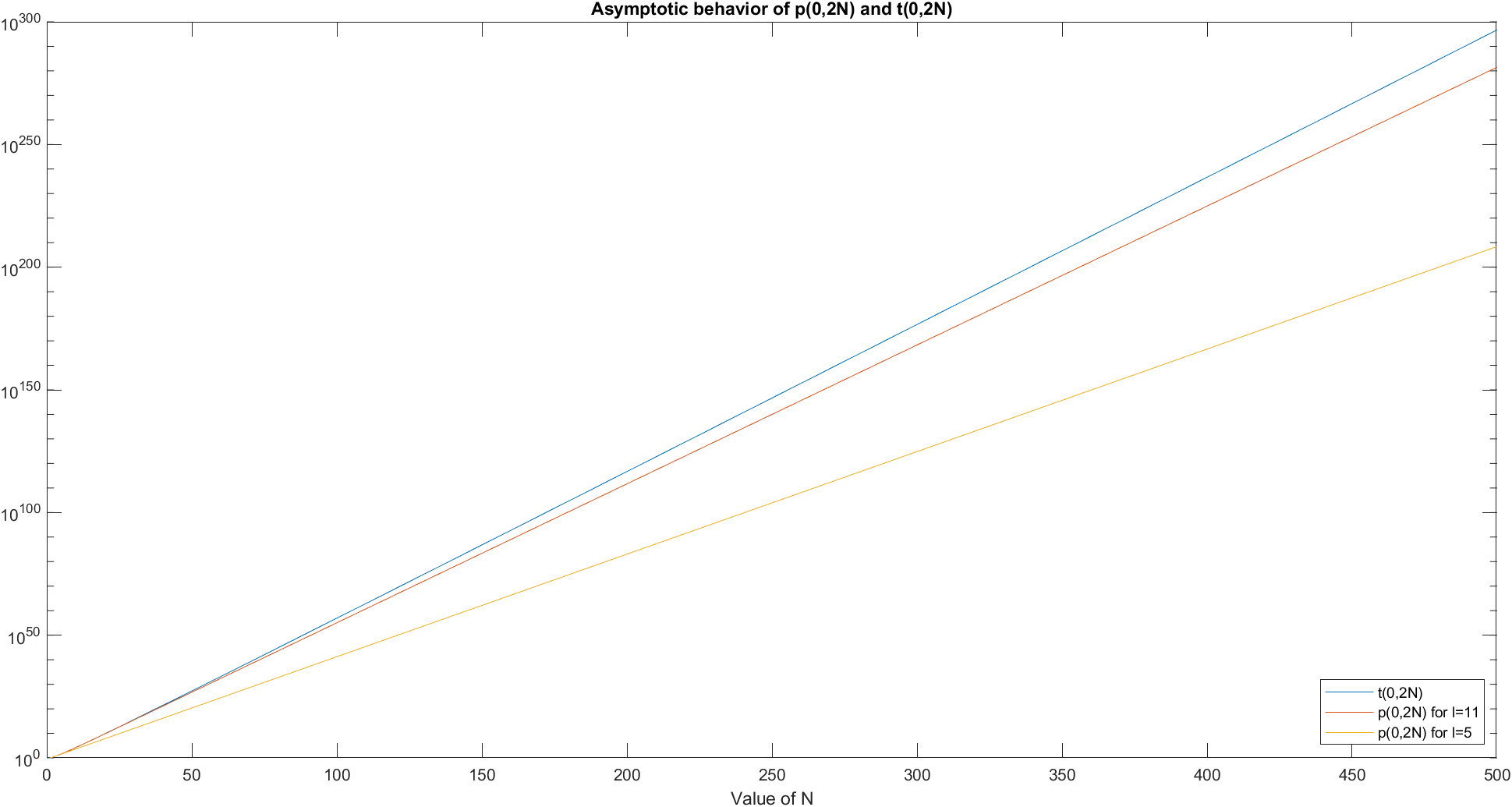}

\bibliography{Bibliography}
\end{document}